\documentclass[11pt, reqno]{amsart}
\usepackage{eurosym}
\usepackage{graphicx,amsmath,amssymb,epsfig, float}
\usepackage{amsfonts, lscape}
\usepackage{color}
\usepackage{bbm,threeparttable,subcaption}
\usepackage{array}
\usepackage{verbatim}
\usepackage{algpseudocode}
\usepackage{adjustbox}
\usepackage{booktabs}

\usepackage{url}
\usepackage{booktabs}
\usepackage{amsfonts}
\usepackage{nicefrac}
\usepackage{microtype}
\usepackage{amsmath}
\usepackage{amsthm,amssymb}
\usepackage[utf8]{inputenc}
\usepackage[T1]{fontenc}
\usepackage{hyperref}
\usepackage{url}
\usepackage{booktabs}
\usepackage{amsfonts}
\usepackage{nicefrac}
\usepackage{adjustbox}
\usepackage{microtype}
\usepackage{algpseudocode}
\usepackage{algorithm}
\usepackage{graphicx}
\usepackage{subcaption}
\usepackage{wrapfig}
\usepackage{makecell}
\usepackage{appendix}

\theoremstyle{plain}

\theoremstyle{definition}

\theoremstyle{definition}

\graphicspath{ {FigsEtc/} }
\newcolumntype{M}[1]{>{\centering\arraybackslash}m{#1}}

\newtheorem{thm}{Theorem}[section]
\newtheorem{rem}[thm]{Remark}

\newtheorem{prop}[thm]{Proposition}
\newtheorem{lem}[thm]{Lemma}

\newcommand{\R}{\mathbb{R}}

\newcommand\tilF{\widetilde {F}}
\newcommand\pref[1]{(\ref{#1})}
\DeclareMathOperator{\argmin}{argmin}

\input{tcilatex}

\begin{document}
\def\sym#1{\ifmmode^{#1}\else\(^{#1}\)\fi}

\title{SISTA: learning optimal transport costs under sparsity constraints}
\author{Guillaume Carlier$^1$, Arnaud Dupuy$^2$, Alfred Galichon$^3$, Yifei Sun$^4$}
\date{\today \\
{
\indent$^1$Department of Applied Mathematics, Universit\'{e} Paris IX Dauphine,\ \texttt{carlier@ceremade.dauphine.fr},\\
\indent$^2$Department of Economics, University of Luxembourg, \texttt{arnaud.dupuy@uni.lu},\\
\indent$^3$Departments of Economics and Courant Institute of Mathematical Sciences, New York University and Department of Economics, Sciences Po, \texttt{galichon@cims.nyu.edu}, Galichon acknowledges support from the European Research Council (ERC) under grant CoG-866274 EQUIPRICE, ``Equilibrium methods for resource allocation and dynamic pricing,''\\
\indent$^4$Courant Institute of Mathematical Sciences, New York University, \texttt{yifei@cims.nyu.edu}}}

\begin{abstract}
In this paper, we describe a novel iterative procedure called SISTA\ to
learn the underlying cost in optimal transport problems. SISTA is a hybrid
between two classical methods, coordinate descent (\textquotedblleft
S\textquotedblright -inkhorn) and proximal gradient descent
(\textquotedblleft ISTA\textquotedblright ). It alternates between a phase
of exact minimization over the transport potentials and a phase of proximal
gradient descent over the parameters of the transport cost. We prove that
this method converges linearly, and we illustrate on simulated examples that
it is significantly faster than both coordinate descent and ISTA. We apply
it to estimating a model of migration, which predicts the flow of migrants
using country-specific characteristics and pairwise measures of
dissimilarity between countries. This application demonstrates the
effectiveness of machine learning in quantitative social sciences.

\vspace{0.1cm} \emph{Keywords:} inverse optimal transport, coordinate descent, ISTA
\end{abstract}

\maketitle

\section{Introduction}

Optimal transport has received a great deal of attention recently, see~\cite%
{villani03},~\cite{villani},~\cite{Santambrogio}, and~\cite{Galichon}. The
discrete version of the problem in a nutshell is as follows. Consider a $%
N\times N$ matrix $c_{ij}$ called a \emph{transport cost}, which is the cost
of pairing $i\in \left\{ 1,...,N\right\} $ with $j\in \left\{
1,...,N\right\} $, and consider two marginal probability distributions $p$
and $q$ with support $\left\{ 1,...,N\right\} $. The optimal transport
problem is the problem of finding the optimal \emph{transport plan}, the
joint probability distribution $\pi _{ij}$ over $\left\{ 1,...,N\right\}
^{2} $ which is optimal in the sense that $\sum_{1\leq i,j\leq N}\pi
_{ij}c_{ij}$ is minimal over all the probability $\pi $ with marginal
distributions $p$ and $q$. A variant of this problem where one seeks to
minimize the cost with a regularizing entropy term, namely $\sum_{1\leq
i,j\leq N}\pi _{ij}c_{ij}+T\pi _{ij}\ln \pi _{ij}$, with $T>0$ being a
temperature parameter, has been proposed by various authors and has been
found to enjoy attractive computational and statistical properties, see e.g.~%
\cite{Idel} and references therein.

Most of the literature on optimal transport takes the transport cost matrix $%
c$ as given and seeks to compute the optimal transport plan $\pi $. However,
in some situations, one \emph{observes }$\pi $ and would like to \emph{learn}
the cost $c$. This problem, which belongs to the general class of
\textquotedblleft inverse problems,\textquotedblright ~see~\cite{Kulis},
arises in particular in economics, see \cite{GalichonSalanie}, where one
observes optimally matched agents and one would like to infer why they have
chosen to match. The problem of learning the transport cost has received
relatively little attention in the machine learning literature. Some
exceptions are~\cite{CuturiAvis}, who use a Difference of Convex (DC)\
functions approach, \cite{Huang}, who use a strategy based on Neighborhood
Components Analysis (NCA), and~\cite{DGS}, who compute a rank-regularized
moment matching estimator of the ground distance using proximal gradient
descent.

In this paper, we assume that the transportation cost $c_{ij}$ takes a
parametric form $c_{ij}^{\beta }$, where the parameter vector $\beta$ is assumed to be sparse. We introduce an
iterative procedure to estimate $\beta $ by minimizing a convex loss
function under an $l_{1}$ penalization based on the dual formulation of the
optimal transport problem. Our procedure is a hybrid between coordinate
descent and proximal gradient descent: indeed, a phase of exact minimization
with respect to the transport potentials alternates with a phase of proximal
gradient descent (ISTA, see \cite{BeckTeboulle}) with respect to the
parameters of the transport cost. This procedure is thus a natural extension
of the celebrated Sinkhorn algorithm (see an account in~\cite{PeyreCuturi}
and a historical overview in \cite{Idel}), a.k.a. the Iterative Proportional
Fitting Procedure (IPFP), which is a coordinate descent algorithm with
respect to the transport potentials, for a fixed value of the parameter
vector. One important advantage of the Sinkhorn algorithm compared with alternative
methods is that it is fast, parameteter-free and can be naturally
parallelized, as documented for example in~\cite{cuturi13} and \cite{BCNP}.

The remainder of the paper is organized as follows. Section \ref%
{sec:formulation} presents the optimal transport problem, recalls duality
results, and introduces the inverse problem of learning the parameter vector
of the transport cost under an $l_{1}$ penalization. Section \ref{sec:SISTA}
describes the SISTA algorithm, states results on its linear convergence, and
benchmarks its speed of convergence via numerical experiments. Section \ref%
{application} applies this procedure to a model of migration based on
country-specific characteristics and pairwise measures of dissimilarity
between countries. Section \ref{sec:conc} concludes the paper.

\section{Optimal transport with entropic regularization}

\label{sec:formulation}

\subsection{Problem formulation}

\label{formulation} Consider two discrete probability distributions $p$ and $%
q$ supported by $N$ points. The optimal transport problem between $p$ and $q$%
, see~\cite{villani03}, is classically defined as%
\begin{equation}
\min_{\pi \in \Pi }\sum_{1\leq i,j\leq N}\pi _{ij}c_{ij}
\label{WassersteinPrimal}
\end{equation}%
where $c_{ij}$ is the transport cost from~$i$ to $j$, and $\Pi $ is the set
of probability mass vectors $\pi $ with margins $p$ and $q$, that is%
\begin{equation}
\Pi =\left\{ \pi _{ij}\geq 0:\sum_{ j=1}^N\pi _{ij}=p_{i}\text{ and }\sum_{
i=1}^N\pi _{ij}=q_{j}\right\} .  \label{defPi}
\end{equation}

Recently, the literature (see for instance \cite{GalichonSalanie} and \cite%
{cuturi13}) has considered the entropic regularization of %
\eqref{WassersteinPrimal}, that is%
\begin{equation}
\min_{\pi \in \Pi }\sum_{1\leq i,j\leq N}\pi _{ij}c_{ij}+T\pi _{ij}\ln \pi
_{ij}  \label{RegulWasserstein}
\end{equation}%
where $T>0$ is a temperature parameter, so that $T\rightarrow 0$ recovers
the previous object. The dual version of~(\ref{RegulWasserstein}) is%
\begin{equation}
\max_{u,v}\sum_{i=1}^{N}p_{i}u_{i}+\sum_{j=1}^{N}q_{j}v_{j}-T\sum_{1\leq i,
j\leq N}\exp \left( \frac{u_{i}+v_{j}-c_{ij}}{T}\right).
\label{RegulWassersteinDualMax}
\end{equation}%
Note that by homogeneity, the solution of problem~(\ref{RegulWasserstein})
is left invariant by dividing $c$ and $T$ by the same constant. Taking this
constant to be $T$, we can without loss of generality set
\begin{equation}
T=1,  \label{TequalsOne}
\end{equation}%
which we will do in the sequel. After taking the negative, we obtain the
following dual optimization problem
\begin{equation}
\min_{u,v}\sum_{1\leq i, j\leq N}\exp \left( u_{i}+v_{j}-c_{ij}\right)
-\sum_{i=1}^{N}p_{i}u_{i}-\sum_{j=1}^{N}q_{j}v_{j}.
\label{RegulWassersteinDual}
\end{equation}

By first order conditions, the optimal $\pi$ in~(\ref{RegulWasserstein}) and
the optimal $\left( u,v\right) $ in~(\ref{RegulWassersteinDual}) are such
that
\begin{equation}
\pi _{ij}=\exp \left( u_{i}+v_{j}-c_{ij}\right) ,  \label{sinkhornBis}
\end{equation}%
where $u$ and $v$ are adjusted so that the primal constraints are satisfied,
i.e. $\pi\in \Pi $ as in~(\ref{defPi}).

The dual problem~(\ref{RegulWassersteinDual}) is a convex minimization
problem, which can be solved by gradient descent. However, it is well-known
that blockwise coordinate descent over $u$ and $v$ iteratively, a procedure
called Sinkhorn's algorithm, see \cite{PeyreCuturi} and references therein,
is a preferable alternative. Given $\left( u^{t},v^{t}\right) $, the
minimization of problem~(\ref{RegulWassersteinDual}) with respect to $u_{i}$
yields, by first order conditions,
\begin{equation}
\exp \left( u_{i}^{t+1}\right) =\frac{p_{i}}{\sum_{j=1}^{N}\exp \left(
v_{j}^{t}-c_{ij}\right) },  \label{ipfpa}
\end{equation}%
while the minimization of~(\ref{RegulWassersteinDual}) with respect to $%
v_{j} $ yields
\begin{equation}
\exp \left( v_{j}^{t+1}\right) =\frac{q_{j}}{\sum_{i=1}^{N}\exp \left(
u_{i}^{t+1}-c_{ij}\right) }.  \label{ipfpb}
\end{equation}

Alternating the minimization steps with respect to $u_{i}$ and $v_{j}$
therefore yields to alternating between the closed-form updating formulas~(%
\ref{ipfpa}) and~(\ref{ipfpb}).

\subsection{Learning the transport cost}

\label{paramCost}

We now turn to the inverse problem of recovering the transport cost $c_{ij}^\beta$
based on the observed transport plan $\hat{\pi}\in \Pi $. To achieve this, we
assume that the cost function can be represented by a linear combination of
basis functions $d^{k}$, $k\in \left\{ 1,...,K\right\} $ as
\begin{equation}
c_{ij}^{\beta }=\sum_{k=1}^{K}\beta _{k}d_{ij}^{k},  \label{paraC}
\end{equation}%
where each $d_{ij}^{k}$ is some measure of dissimilarity between $i$ and $j$%
, and $\beta =(\beta
_{1},\cdots ,\beta _{K})^{\top } $ is a parameter vector to be learned.

We present a few choices of dissimilarity measures $d_{ij}^{k}$.

\textbf{(a)}. If $x^i$ and $y^j$ are vectors of characteristics for $i$ and $%
j$, one may set $d_{ij}^{k} =\left( x^i_k-y^j_k\right) ^{2}$, where the
subscript $k$ denotes the $k$-th component of a vector.

\textbf{(b)}. Similarly, if $k=(r,s)$ stands for a pair of indices, one may
set $d_{ij}^{rs}=\left( x_{r}^{i}-y_{s}^{j}\right) ^{2}$ as is done in~\cite%
{dupuy}, in order to capture off-diagonal interactions between
characteristics. The matrix parameter $\beta _{rs}$ to be learned measures
the interaction between two distinct characteristics.

\textbf{(c)}. If, like in our application, $i$ and $j$ are countries, then $%
d_{ij}^{k}$ may be the share of inhabitants of country $i$ who do not speak
the language of country $j$. Note that this measure has no reason to be
symmetric since $d_{ij}^{k}\neq d_{ji}^{k}$ in general.

Our learning procedure is based on looking for $\beta $ such that $\pi
_{ij}^{\beta }=\exp \left( u_{i}+v_{j}-c_{ij}^{\beta }\right) $ solves the
margin constraints $\pi ^{\beta }\in \Pi $, and matches the moments of $%
d^{k} $ for $k=1,...,K$, that is
\begin{equation}
\sum_{i=1}^{N}\pi _{ij}^{\beta }=q_{j},\sum_{j=1}^{N}\pi _{ij}^{\beta
}=p_{i},\sum_{1\leq i,j\leq N}\pi _{ij}^{\beta }d_{ij}^{k}=\sum_{1\leq
i,j\leq N}\hat{\pi}_{ij}d_{ij}^{k}.  \label{foc}
\end{equation}

As shown by~\cite{GalichonSalanie}, equations~(\ref{foc}) arise as the first
order conditions of the following minimization problem
\begin{equation}
\min_{u,v,\beta }F\left( u,v,\beta \right) ,  \label{lossMin}
\end{equation}%
where%
\begin{equation}
F\left( u,v,\beta \right) :=\sum_{1\leq i,j\leq N}\exp \left(
u_{i}+v_{j}-c_{ij}^{\beta }\right) +\sum_{1\leq i,j\leq N}\hat{\pi}%
_{ij}\left( c_{ij}^{\beta }-u_{i}-v_{j}\right)  \label{deff}
\end{equation}%
is a convex function.

Now assuming that $\beta $, the parameter vector, is sparse, one way to
handle this problem is to add the $l_{1}$ penalty term in~(\ref{lossMin}),
namely considering the following problem:
\begin{equation}
\min_{u,v,\beta }\Phi \left( u,v,\beta \right) :=F\left( u,v,\beta \right)
+\gamma \left\vert \beta \right\vert _{1},  \label{LassoMLE}
\end{equation}%
where $\gamma $ is a regularization parameter. This regularization is
similar in spirit to~\cite{DGS}, who have introduced a rank regularization
technique using the nuclear norm in a problem where $\beta $ is a matrix. In
the next section, we provide the SISTA algorithm for the computation of~(\ref%
{LassoMLE}).

\section{SISTA}

\label{sec:SISTA}

\subsection{Algorithm}

The SISTA algorithm consists of the minimization of $\Phi \left( u,v,\beta
\right) $ in~(\ref{LassoMLE}) by iterating three alternating phases:

\begin{itemize}
\item an exact minimization of with respect to $u$, holding $\beta $ and $v$
constant;

\item an exact minimization of with respect to $v$, holding $\beta $ and $u$
constant;

\item a proximal gradient descent step with respect to $\beta $, holding $u $
and $v$ constant.
\end{itemize}

Each of these three phases is straightforward. The exact minimization with
respect to $u$ is done in closed-form by the updating formula~(\ref{ipfpa})
and for $v$ done by (\ref{ipfpb}). The proximal gradient descent step with
respect to $\beta $ yields%
\begin{equation*}
\beta ^{t+1}=\text{prox}_{\rho \gamma |\cdot |_{1}}\left( \beta ^{t}-\rho
\nabla _{\beta }F\left( u^{t+1},v^{t+1},\beta ^{t}\right) \right) ,
\end{equation*}%
which is explicit and whose components are given by the well-known
soft-thresholding formula:
\begin{equation*}
\text{prox}_{\rho \gamma |\cdot |}(z)=%
\begin{cases}
z-\rho \gamma & \text{if }z>\rho \gamma \\
0 & \text{if }|z|\leq \rho \gamma \\
z+\rho \gamma & \text{if }z<-\rho \gamma%
\end{cases}%
.
\end{equation*}

Combining these three steps yields SISTA, which we describe in algorithm \ref%
{proxAlgo}.

\begin{algorithm}
\caption{The SISTA algorithm}\label{proxAlgo}
\begin{algorithmic}
\Require Initial guess of parameter vector $\beta ^{0}$, of potentials $u^0$ and $v^0$, step size $\rho $, and dissimilarity measures $d_{ij}^{k}$%
\While {not converged}
	\State (Sinkhorn step). Set $c^{\beta^t}_{ij}:=\sum_{k=1}^K\beta _{k}^{t}d_{ij}^{k}$ and update:
\begin{equation*}
\left\{
\begin{array}{l}
\exp \left( u_{i}^{t+1}\right) =\frac{p_i}{\sum_{j=1}^N\exp \left(
v_{j}^{t}-c_{ij}^{\beta^t}\right) } \\
\exp \left( v_{j}^{t+1}\right) =\frac{q_j}{\sum_{i=1}^N\exp \left(
u_{i}^{t+1}-c_{ij}^{\beta^t}\right)}
\end{array}%
\right.
\end{equation*}
	\State (ISTA step). Let\ $\pi _{ij}^{\beta^t}:=\exp \left(
u_{i}^{t+1}+v_{j}^{t+1}-c_{ij}^{\beta^t}\right) $. For $k=1,\ldots, K,$\State%
\begin{equation*}
\beta ^{t+1}_k=\text{prox}_{\rho\gamma|\cdot|}\Big(\beta ^{t}_k-\rho \sum_{1\leq i,j\leq N}\Big( \hat{\pi}_{ij}-\pi _{ij}^{\beta ^{t}}
\Big) d_{ij}^{k}\Big).
\end{equation*}

\EndWhile
\Ensure $\beta$
\end{algorithmic}
\end{algorithm}

\subsection{Convergence}

We introduce two assumptions on the dissimilarity measures $d_{ij}^{k}$:
\begin{equation}
\forall k=1,\ldots ,K,\;\sum_{i=1}^{N}d_{ij}^{k}=0,\;\forall
j,\;\sum_{j=1}^{N}d_{ij}^{k}=0,\forall i,  \label{ass1}
\end{equation}%
and
\begin{equation}
\mbox{the matrices }\{d^{1},\ldots ,d^{K}\}\mbox{ are linearly independent}.
\label{ass2}
\end{equation}

Note that (\ref{ass1}) is without loss of generality, see section \ref{prelim}. In
addition, we assume
\begin{equation}
\hat{\pi}_{ij}>0,\;\forall (i,j)\in \{1,\ldots ,N\}^{2}.  \label{ass3}
\end{equation}%
We are now ready to state our main theorem of the paper:

\begin{thm}
\label{cvsistathm} \label{cvhyb} Assume (\ref{ass1})-(\ref{ass2})-(\ref{ass3}%
). The sequence $x^t:= (u^t, v^t, \beta^t)$ generated by the SISTA scheme in
algorithm \ref{proxAlgo} converges to the solution $x^*=(u^*, v^*, \beta^*)$
of (\ref{LassoMLE}) as $t\to +\infty$ provided the step $\rho>0$ is chosen
small enough. Moreover, in this case, there exists $\delta>0$ such that
\begin{equation}  \label{linearcvsista}
\Phi(x^t)-\Phi(x^*) \leq \frac{\Phi(x^0)-\Phi(x^*) }{(1+\delta)^t}, \;
\forall t\in {\mathbb{N}}.
\end{equation}
\end{thm}

\subsection{Proof of convergence} 
We present a full proof of Theorem \ref{cvsistathm} in this section. 
\subsubsection{Preliminaries} \label{prelim}

Our aim is to solve the convex problem 
\begin{equation}
\inf_{(u, v, \beta)\in \mathbb{R}^N \times \mathbb{R}^N\times \mathbb{R}^K}
\Phi(u, v, \beta):=G\circ \Lambda(u,v, \beta)+ \gamma \vert \beta\vert_1,
\end{equation}
where $\Lambda$ is the linear map $\mathbb{R}^N \times \mathbb{R}^N\times 
\mathbb{R}^K \to \mathcal{M}_N$ defined entrywise by 
\begin{equation}  \label{defLam}
(\Lambda(u, v, \beta))_{ij}:= u_i + v_j -c_{ij}^\beta, \; \forall (i, j),
\end{equation}
and $G$ is the (smooth, strictly convex and separable) function defined by: 
\begin{equation}
G(\lambda):= \sum_{1\leq i,j \leq N} ( \exp(\lambda_{ij})- \hat{\pi}_{ij}
\lambda_{ij}), \; \forall \lambda \in \mathcal{M}_N.
\end{equation}
Since 
\begin{equation}  \label{derivG}
\nabla G(\lambda)=\exp({\lambda})-\hat{\pi}, \; D^2 G(\lambda)={\mathrm{diag}%
}(\exp(\lambda))
\end{equation}
(where $\exp({\lambda})$ denotes the matrix with entries $\exp(\lambda_{ij})$%
, and ${\mathrm{diag}}(\exp(\lambda))$ is the $N^2\times N^2$ diagonal matrix
having entries $\exp(\lambda_{ij})$ on its diagonal), we also have (setting $%
x=(u,v, \beta)$ to shorten notations): 
\begin{equation}  \label{derivF}
\nabla F(x)=\Lambda^T (\exp{\Lambda(x)}-\hat{\pi}), \; D^2 F(x)= \Lambda^T {%
\mathrm{diag}}(\exp(\Lambda(x))) \Lambda.
\end{equation}
Properties of $F$ of course strongly rely on properties of $\Lambda$.
Recall the two assumptions (\ref{ass1}) and (\ref{ass2}): 
\begin{equation*}  
\forall k=1, \ldots,K, \; \sum_{i=1}^N d_{ij}^k=0, \; \forall j, \;
\sum_{j=1}^N d_{ij}^k=0, \forall i,
\end{equation*}
and 
\begin{equation*}  
\mbox{the matrices } \{d^1, \ldots, d^K\} \mbox{ are linearly independent}.
\end{equation*}
Note that (\ref{ass1}) is not really a restiction. Indeed, if we introduce
the new matrices $\widetilde {d}^k$ 
\begin{equation*}
\widetilde {d}_{ij}^k:= d_{ij}^k-a_i^k-b_j^k
\end{equation*}
where 
\begin{equation*}
a_i^k:=\frac{1}{N} \sum_{j=1}^N d_{ij}^k, \; b_j^k:=\frac{1}{N} \sum_{i=1}^N
d_{ij}^k-\frac{1}{N^2} \sum_{1\leq p, q \leq N} d_{p,q}^k
\end{equation*}
then, obviously $\{\widetilde {d}^1, \ldots, \widetilde {d}^K\}$ satisfies
the row and columns zero sum conditions from (\ref{ass1}). Defining 
\begin{equation*}
\widetilde {c}^\beta:=\sum_{k=1}^K \beta_k \widetilde {d}^k, \; \forall
\beta\in \mathbb{R}^K
\end{equation*}
as well as the linear map, $\widetilde {\Lambda}$: 
\begin{equation*}
\widetilde {\Lambda}(\widetilde {u}, \widetilde {v}, \beta)_{ij}:=%
\widetilde {u}_i+\widetilde {v}_j-\widetilde {c}_{ij}^\beta,
\end{equation*}
we immediately see that $\Lambda(u, v, \beta)=\widetilde {\Lambda}%
(u-\sum_{k=1}^K \beta_k a^k, v-\sum_{k=1}^K \beta_k b^k, \beta)$. In other
words, there is no loss of generality in replacing the matrices $d^k$ by the
matrices $\widetilde {d}^k$ in our minimization problem. Observing finally
that for every constant vector $m$, $\Lambda(u+m, v-m, \beta)=\Lambda(u, v,
\beta)$, we can remove this invariance by imposing a normalization conditon
on $u$ or $v$, we will impose the simplest one, namely $u_1=0$. Let us then
set 
\begin{equation}  \label{defdeE}
E:=\{(u,v, \beta)\in \mathbb{R}^N\times \mathbb{R}^N\times \mathbb{R}^K \; :
\; u_1=0\}\simeq \mathbb{R}^{2N-1+K}.
\end{equation}

\begin{lem}
\label{injlem} Assume (\ref{ass1})-(\ref{ass2}), then $\Lambda$ is an
injective map from $E$ to $\mathcal{M}_N$. In particular $\Phi$ is strictly
convex on $E$.
\end{lem}

\begin{proof}
Let $(u,v, \beta)\in E$ be in the null space of $\Lambda$ i.e.
\begin{equation}\label{nullspacel}
u_i+v_j=\sum_{k=1}^K \beta_k d_{ij}^k, \forall i,j.
\end{equation}
Summing over $j$ and using \pref{ass1} give $Nu_i+\sum_{j=1}^N v_j=0$ and then taking $i=1$ gives $\sum_{j=1}^N v_j=0$ hence $u=0$ and $v=0$ as well. Therefore \pref{nullspacel} becomes $\sum_{k=1}^K  \beta_k d^k=0$ so that $\beta=0$ thanks to \pref{ass2}.
\end{proof}

From now on, we will assume that (\ref{ass1})-(\ref{ass2}) hold. Then,
thanks to Lemma \ref{injlem} and formulas (\ref{derivG})-(\ref{derivF}), we
see that for every $M>0$ there exist $\nu =\nu (M)>0$ and $\alpha =\alpha
(M) $ such that whenever $x:=(u,v,\beta )$ and $y:=(u^{\prime },v^{\prime
},\beta ^{\prime })$ are in $E$ and $\max (\Vert \Lambda (u,v,\beta )\Vert
_{\infty },\Vert \Lambda (u^{\prime },v^{\prime },\beta ^{\prime })\Vert
_{\infty })\leq M$, one has the ellipticity condition 
\begin{equation}
F(x)\geq F(y)+\nabla F(y)(x-y)+\frac{\nu }{2}\Vert x-y\Vert _{2}^{2}
\label{eqelliptloc}
\end{equation}%
(one can just take $\nu (M):=e^{-M}\sigma _{\min }$ where $\sigma _{\min }>0 
$ is the smallest eigenvalue of $\Lambda ^{T}\Lambda $) as well the
Lipschitz estimate 
\begin{equation}  \label{LipgFloc}
\Vert \nabla F(x)-\nabla F(y)\Vert _{2}\leq \alpha \Vert x-y\Vert _{2},%
\mbox{ with }\alpha =e^{M}\Vert \Lambda \Vert _{2}\Vert \Lambda ^{T}\Vert
_{2},
\end{equation}
where $\Vert A\Vert _{2}$ denotes the $2$-operator norm of a matrix $A$.

Our last assumption (\ref{ass3}) is 
\begin{equation*} 
\hat{\pi}_{ij}>0 , \; \forall (i,j)\in \{1, \ldots, N\}^2.
\end{equation*}
Rewriting our initial optimization problem with the normalization constraint 
$u_1=0$ as 
\begin{equation}  \label{minifi}
\inf_{(u,v, \beta)\in \mathbb{R}^{2N-1+K}} \Phi(u, v, \beta)=G(\Lambda(u, v,
\beta))+ \gamma \vert \beta \vert_1
\end{equation}
we then have:

\begin{prop}
\label{propminifi} Under assumptions (\ref{ass1})-(\ref{ass2})-(\ref{ass3}),
we have:

\begin{itemize}
\item $\Phi$ is coercive on $\mathbb{R}^{2N-1+K}$, i.e. its sublevels are
bounded,

\item problem (\ref{minifi}) admits a unique solution $x^*:=(u^*, v^*,
\beta^*)$

\item $x^*:=(u^*, v^*, \beta^*)$ is characterized by the optimality
conditions 
\begin{equation}  \label{opticond}
\nabla_ u F(x^*)=0, \; \nabla_v F(x^*)=0, \; -\nabla_{\beta} F(x^*) \in
\gamma \partial \vert. \vert_{1} (\beta^*).
\end{equation}
\end{itemize}
\end{prop}

\begin{proof}
Observe that whenever $p>0$ one has $e^t-p t\geq p-p \log(p)$ for every $t\in \R$ (Young's inequality). Obviously, if $\lambda<0$ we have $e^\lambda-\hat{\pi}_{ij} \lambda \geq \hat{\pi}_{ij} \vert \lambda \vert$. If $\lambda\geq 0$, $e^\lambda-\hat{\pi}_{ij} \lambda= \hat{\pi}_{ij} \vert \lambda \vert +e^\lambda- 2\hat{\pi}_{ij} \lambda$, using Young's inequality and the fact that $\hat{\pi}_{ij} \in [0,1]$, we get
\[e^\lambda-\hat{\pi}_{ij} \lambda \geq  \hat{\pi}_{ij} \vert \lambda \vert-2 \log(2).\]
Hence
\begin{equation}\label{coercexpl}
\Phi(u, v, \beta)\geq \sum_{1\leq i,j\leq N} \hat{\pi}_{ij} \vert \Lambda(u,v,\beta)_{ij} \vert -2N^2 \log(2)+\gamma \vert \beta\vert_1.
\end{equation}
Thanks to \pref{ass3} and the injectivity of $\Lambda$,  the coercivity of $\Phi$ follows. Since $\Phi$ is continuous, there therefore exists a solution for \pref{minifi}  which is unique by strict convexity. Since $\Phi$ is the sum of the  smooth term $F$ and the $l_1$ norm of the last component, it is easy to see that the optimality condition $0\in \partial \Phi(u^*, v^*, \beta^*)$ is exactly the system \pref{opticond}.
\end{proof}

\subsubsection{Proof of convergence}

Starting from $(u^0, v^0, \beta^0)\in \mathbb{R}^{2N-1+K}$, inductively
define the sequence $(u^t, v^t, \beta^t)$ via: 
\begin{equation}  \label{cd1}
u^{t+1}=\argmin F(., v^t, \beta^t), \; v^{t+1}=\argmin F(u^{t+1}, ., \beta^t).
\end{equation}
Note that these coordinate descent updates are explicit: 
\begin{equation*}
\exp(u_i^{t+1})=\frac{p_i}{ \sum_{j=1}^N \exp(v_j^t-c_{ij}^{\beta^t})},
\exp(v_j^{t+1})=\frac{q_j}{ \sum_{i=1}^N \exp(u_i^{t+1}-c_{ij}^{\beta^t})}.
\end{equation*}

We now consider the SISTA hybrid method where the updates for $\beta$ are
given by the ISTA algorithm (see [1]) with a constant step $\rho>0$: 
\begin{equation}  \label{ista}
\beta^{t+1}=\mathop{\mathrm{prox}}\nolimits_{\rho \gamma \vert .\vert_{1}} %
\Big (\beta^t-\rho \nabla_\beta F(u^{t+1}, v^{t+1}, \beta^t)\Big)
\end{equation}
i.e. $\beta^{t+1}$ is obtained by minimizing with respect to $\beta\in 
\mathbb{R}^K$ 
\begin{equation}  \label{prox}
\beta \mapsto \rho \gamma \vert \beta \vert_{1} + \frac{1}{2} \Vert
\beta-(\beta^t-\rho \nabla_\beta F(u^{t+1}, v^{t+1}, \beta^t))\Vert_2^2.
\end{equation}
Note that the solution of (\ref{prox}) is explicit and given by a well-known
thresholding formula. So all the steps for algorithm (\ref{cd1})-(\ref{ista}%
) are totally explicit. The step size $\rho$ has to be chosen in an
appropriate way to ensure convergence.
Setting $C=\Phi(u^0, v^0, \beta^0)$, thanks to \pref{coercexpl} we have
\[\Phi(u, v, \beta) \leq C \Rightarrow \vert \beta \vert_1 \leq R:=\frac{C+ 2N^2 \log(2)}{\gamma}. \]
Defining $A:=C+ 2N^2 \log(2)+1$ we thus have
\begin{equation}\label{propdeA}
A> \gamma R, \mbox{ and  $F\leq A$ whenever $\Phi \leq C$}.
\end{equation}
Now let $\theta \in C^{\infty}(\R, \R)$ be a nondecreasing function such that
\begin{equation}\label{propdetheta}
\theta(t)=t \mbox{ when $t\leq A$}, \; \theta(t)\geq t \mbox{ when $t\in [A, 2A]$} \mbox{ and } \theta(t)=2A \mbox{ when $t\geq 2A$}.
\end{equation}
The function  $\tilF:=\theta \circ F$ is nonconvex but its gradient is globally Lipschitz. Let $\alpha>0$ be a Lipschitz constant of $\nabla \tilF$ (and of $\nabla F$ on $F\leq A$), and now take our step $\rho$ so as to satisfy
\begin{equation}\label{choicero}
0 <\rho \leq \frac{1}{\alpha}.
\end{equation}
Note that since $\nabla \tilF$ is $\alpha$-Lipschitz, we have
\begin{equation}\label{semiconcave}
\tilF(x) \leq \tilF(y)+ \nabla \tilF(y)\cdot(x-y)+\frac{\alpha}{2} \Vert x-y\Vert^2_2, \; \forall (x,y)\in E^2.
\end{equation}
Let us now show inductively that  for $\rho$ satisfying \pref{choicero}, the iterates \pref{cd1}-\pref{ista} remain in the sublevel set $\Phi \leq C$. Of course, if $\Phi(u^t, v^t, \beta^t)\leq C$ then by \pref{cd1}, $C \geq \Phi(u^{t+1}, v^t, \beta^t)\geq \Phi(u^{t+1}, v^{t+1}, \beta^t)$. Proving that $\Phi(u^{t+1}, v^{t+1}, \beta^{t+1})\leq C$ requires a little more work. First note that 
\[F(u^{t+1}, v^{t+1}, \beta^t)=\tilF(u^{t+1}, v^{t+1}, \beta^t) \mbox{ and } \nabla F(u^{t+1}, v^{t+1}, \beta^t)=\nabla \tilF(u^{t+1}, v^{t+1}, \beta^t).\]
Now, it follows from \pref{ista} that $\beta^{t}-\beta^{t+1}-\rho \nabla_{\beta} F(u^{t+1}, v^{t+1}, \beta^t) \in \rho \gamma \partial (\vert . \vert_{1})(\beta^{t+1})$, hence
\begin{equation}\label{istaineq1}
\gamma \vert \beta^{t} \vert_{1} \geq \gamma \vert \beta^{t+1} \vert_{1}+ \frac{1}{\rho} \Vert \beta^{t+1}-\beta^t\Vert^2_2 - \nabla_\beta F(u^{t+1}, v^{t+1}, \beta^t) \cdot(\beta^t-\beta^{t+1})
\end{equation}
but thanks to \pref{semiconcave} we also have
\[\begin{split}
F(u^{t+1}, v^{t+1}, \beta^t) &\geq \tilF(u^{t+1}, v^{t+1}, \beta^{t+1})\\ &+ \nabla_\beta F(u^{t+1}, v^{t+1}, \beta^t) \cdot(\beta^t-\beta^{t+1})-\frac{\alpha}{2} \Vert \beta^{t+1}-\beta^t\Vert^2_2.
\end{split}\]
Summing this inequality with \pref{istaineq1} and using \pref{choicero}, we thus get
\begin{equation}\label{istaineq2}
\Phi(u^{t+1}, v^{t+1}, \beta^t) \geq \tilF(u^{t+1}, v^{t+1}, \beta^{t+1})+\gamma \vert \beta^{t+1} \vert_{1} + \frac{\alpha}{2} \Vert \beta^{t+1}-\beta^t\Vert^2_2.
\end{equation}
If $F(u^{t+1}, v^{t+1}, \beta^{t+1}) \leq 2A$ then by \pref{propdetheta} we have $\tilF(u^{t+1}, v^{t+1}, \beta^{t+1}) \geq F(u^{t+1}, v^{t+1}, \beta^{t+1})$ and therefore $\Phi(u^{t+1}, v^{t+1}, \beta^{t+1}) \leq C$. If, on the contrary, $F(u^{t+1}, v^{t+1}, \beta^{t+1}) > 2A$ then \pref{istaineq2} would imply that $2A \leq F(u^{t+1}, v^{t+1}, \beta^t) +\gamma \vert \beta^t  \vert_{1}\leq A + \gamma R$ which contradicts our choice of $A$ in \pref{propdeA}. This shows that
\begin{equation}\label{istaineq3}
\Phi(u^{t+1}, v^{t+1}, \beta^t) \geq   \Phi(u^{t+1}, v^{t+1}, \beta^{t+1}) + \frac{\alpha}{2} \Vert \beta^{t+1}-\beta^t\Vert^2_2
\end{equation}
and in particular this proves the desired conclusion that $ \Phi(u^{t+1}, v^{t+1}, \beta^{t+1})\leq C$.  The iterates of SISTA therefore remain bounded and the value of  $\Phi$ along these iterates converges monotonically. In particular, these iterates remain in a ball where both the uniform ellipticity condition \pref{eqelliptloc} and the Lipschitz estimate \pref{LipgFloc} on $\nabla F$ are satisfied. Using \pref{eqelliptloc} together with $\nabla_u F(u^{t+1}, v^{t}, \beta^t)=0$ , $\nabla_v F(u^{t+1}, v^{t+1}, \beta^t)=0$, we get%
\begin{eqnarray*}
F(u^{t},v^{t},\beta ^{t})-F(u^{t+1},v^{t},\beta ^{t}) &\geq &\frac{%
\nu }{2} \Vert u^{t+1}-u^{t}\Vert _{2}^{2} \\
F(u^{t+1},v^{t},\beta ^{t})-F(u^{t+1},v^{t+1},\beta ^{t})&\geq &%
\frac{\nu }{2} \Vert v^{t+1}-v^{t}\Vert_{2}^{2}.
\end{eqnarray*}%
Summing and using $F(u^t, v^t, \beta^t)-F(u^{t+1}, v^{t+1}, \beta^t)=\Phi(u^t, v^t, \beta^t)-\Phi(u^{t+1}, v^{t+1}, \beta^t)$, we deduce
\begin{equation}\label{telfi1sista}
\Phi(u^t, v^t, \beta^t)-\Phi(u^{t+1}, v^{t+1}, \beta^t)\geq \frac{\nu}{2} ( \Vert u^{t+1}-u^t \Vert_2^2 +   \Vert v^{t+1}-v^t \Vert_2^2).
\end{equation}
Together with \pref{istaineq3}, this gives
\begin{equation*}
\Phi(u^t, v^t, \beta^t)-\Phi(u^{t+1}, v^{t+1}, \beta^{t+1}) \geq \frac{\nu}{2} ( \Vert u^{t+1}-u^t \Vert_2^2 +   \Vert v^{t+1}-v^t \Vert_2^2) + \frac{\alpha}{2} \Vert \beta^{t+1}-\beta^t \Vert_2^2.
\end{equation*}
Setting $A_t:=\Phi(x^t)-\Phi(x^*)$, using that $\alpha \geq \nu$, we thus have:
\begin{equation}\label{estimsis1}
A_{t-1} -A_{t} \geq \frac{\nu}{2} \Vert x^t-x^{t-1}\Vert_2^2,  \; \forall t\geq 1.
\end{equation}
Let us now bound $A_t$ from above. By construction of $\beta^t$ by proximal gradient descent, we have
\begin{equation*}
 q_t:=\frac{\beta^{t-1}-\beta^t}{\rho}-\nabla_{\beta} F(u^t, v^t, \beta^{t-1}) \in \gamma \partial (\vert . \vert_{1}) (\beta^t).
  \end{equation*}
 Using  \pref{eqelliptloc} and the fact that $q_t \in \gamma \partial( \vert . \vert_{1}) (\beta^t)$, we thus obtain
\[\begin{split}
\Phi(x^*)&\geq  \Phi(x^t) + \nabla_u F(x^t) \cdot (u^*-u^t)+\frac{\nu}{2} \Vert u^t-u^*\Vert_2^2\\
&+ \nabla_v F(x^t) \cdot (v^*-v^t)+\frac{\nu}{2} \Vert v^t-v^*\Vert_2^2\\
&+(\nabla_\beta  F(x^t) +  q_t) \cdot (\beta^*-\beta^t)+\frac{\nu}{2} \Vert \beta^t-\beta^*\Vert_2^2.
\end{split}\]
Using Young's inequality: $q\cdot z +\frac{\nu}{2} \Vert z \Vert_2^2 \geq -\frac{1}{2\nu} \Vert q \Vert_2^2$,  this yields
\begin{equation}\label{estimsis2}
A_t \leq \frac{1}{2\nu} \Big(  \Vert \nabla_u F(x^t)\Vert_2^2+  \Vert \nabla_v F(x^t)\Vert_2^2 +  \Vert \nabla_\beta F(x^t) +q_t\Vert_2^2 \Big).
\end{equation}
Now since $\nabla_u F(u^t, v^{t-1}, \beta^{t-1})=0$ and $\nabla_v F(u^t, v^{t}, \beta^{t-1})=0$,  thanks to \pref{LipgFloc}, we have  
\[ \Vert \nabla_u F(x^t)\Vert_2^2+  \Vert \nabla_v F(x^t)\Vert_2^2 \leq  2\alpha^2 \Vert x^t-x^{t-1}\Vert^2_2.\]
 Thanks  to \pref{LipgFloc}, the monotonicity of $\nabla_{\beta} F(u^t, v^t, .)$ and the definition of $q_t$, we have
\[\begin{split} 
\Vert \nabla_\beta F(x^t) +q_t\Vert_2^2 &= \bigg\Vert \frac{\beta^{t-1}-\beta^t}{\rho}+  \nabla_\beta F(x^t)-\nabla_{\beta} F(u^t, v^t, \beta^{t-1})\bigg\Vert^2_2\\
&\leq  ( \rho^{-2}+\alpha^2) \Vert \beta^{t-1}-\beta^t\Vert_2^2\\
& + \frac{2}{\rho} (\beta^{t-1}-\beta^t)\cdot (\nabla_\beta F(x^t)-\nabla_{\beta} F(u^t, v^t, \beta^{t-1}))\\
& \leq ( \rho^{-2}+\alpha^2)  \Vert \beta^{t-1}-\beta^t\Vert_2^2.
\end{split}\]
Using these inequalities in \pref{estimsis2}, we deduce that
\begin{equation}\label{estimsis3}
A_t \leq \frac{3 \alpha^2+ \rho^{-2}}{2 \nu}\Vert x^t-x^{t-1}\Vert^2_2,
\end{equation}
which, combined with \pref{estimsis1}, gives the desired linear convergence rate:
\[A_t \leq \frac{A_{t-1}}{1+\delta} \mbox{ with } \delta= \frac{\nu^2 \rho^2}{ 3\alpha^2 \rho^2+1}.\]

\begin{rem}
Since $\Vert x^t-x^{t+1}\Vert_2 \leq \sqrt{ \frac{{2(\Phi(x^t)-\Phi(x^*))}}{%
\nu}}$, linear convergence of $\Phi(x^t)-\Phi(x^*)$ together with the
triangle inequality straightforwardly gives $\Vert x^t-x^{*}\Vert_2=
O((1+\delta)^{-\frac{t}{2}})$.
\end{rem}

\subsubsection{Discussion}
We point out a few remarks that are useful in applying SISTA\ in practice.

\begin{itemize}
\item It is clear from the proof above that using varying steps $\rho_t$
that are bounded away from 0 and satisfy \eqref{choicero} leads to the same
convergence results, which also allows in practice to use line search
methods.

\item Assumption \eqref{ass3} may look restrictive in some applications.
However, all of our analysis can be generalized to cases where some of the $%
\hat{\pi}_{ij}$ are zero. In those cases, setting $I^{+}:=\{(i,j):\hat{\pi}%
_{ij}>0\}$, $F$ should just be replaced by the corresponding sum over $I^{+}$%
. In the context of migration we study in our application, for example, the
number of migrants from country $i$ to country $j$ is only defined if $i\neq
j$.

\item Our method also applies to the case where one imposes sign constraints
on some of the $\beta_k$'s. Indeed, the corresponding proximal operator is
explicit as well. More generally, SISTA can be used for any \emph{simple}
(i.e. having a closed-form prox) penalization, for instance, the group lasso.
\end{itemize}

\subsection{Numerical simulations}

We demonstrate the fast convergence of SISTA in practice on simulated
examples varying $K$ and $N$, the dimensionality of $\beta $ and size of the
margins of $\pi $, respectively. We compare against two other algorithms:
ISTA and coordinate descent. In ISTA, we perform gradient descent on $u$, $v$
(as the objective function is smooth for these parameters), and ISTA on $%
\beta $, which is the only part of the variables subject to a nonsmooth
penalization. For coordinate descent, we apply Sinkhorn minimization to $u$,
$v$ and univariate coordinate descent (using the bisection method) to each
component of $\beta $.

We generate each $d_{ij}^{k}$ from an i.i.d. standard normal distribution.
We additionally draw each $\hat{\pi}_{ij}$ from an i.i.d. standard log
normal distribution. For each problem size, we choose $\gamma$ such that the
sparsity level, defined as the number of non-zero $\beta$ components divided
by $K$, is at 0.05 or 0.1. We run SISTA on the simulated data with high
precision to obtain $\left( u^{\ast },v^{\ast },\beta ^{\ast }\right) $ and
then plot $\left\vert \Phi \left( u^{t},v^{t},\beta ^{t}\right) -\Phi \left(
u^{\ast },v^{\ast },\beta ^{\ast }\right) \right\vert $ against $T\left(
t\right) $, the computation time at iteration $t$ in seconds, in a log-log
plot, as illustrated in figure \ref{simulations}.

\begin{figure}[t]
\centering
\begin{subfigure}[t]{0.26\textwidth}
\centering
\includegraphics[width=\linewidth]{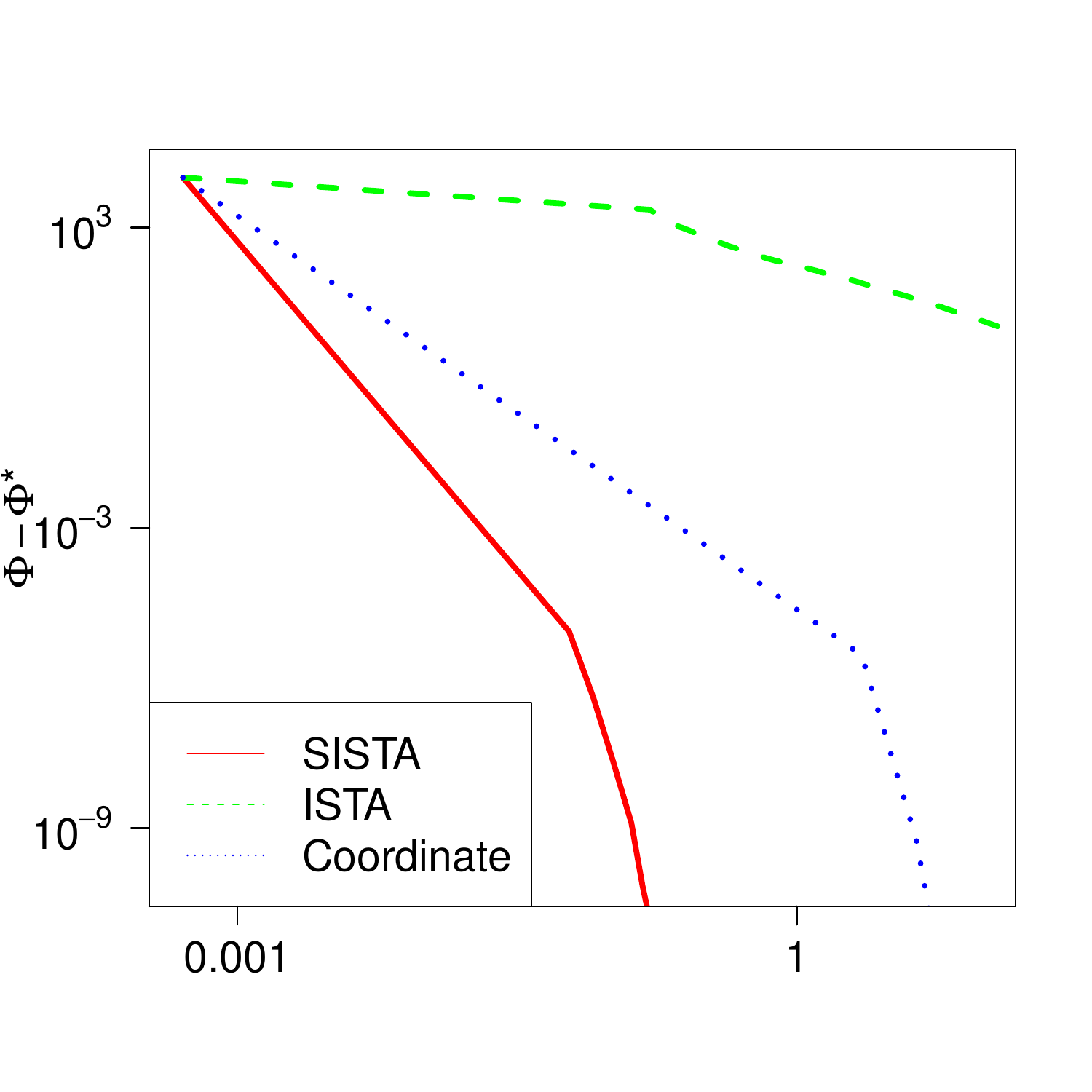}
\end{subfigure}\hspace{-0.15in} \vspace{-0.3in}
\begin{subfigure}[t]{0.26\textwidth}
\centering
\includegraphics[width=\linewidth]{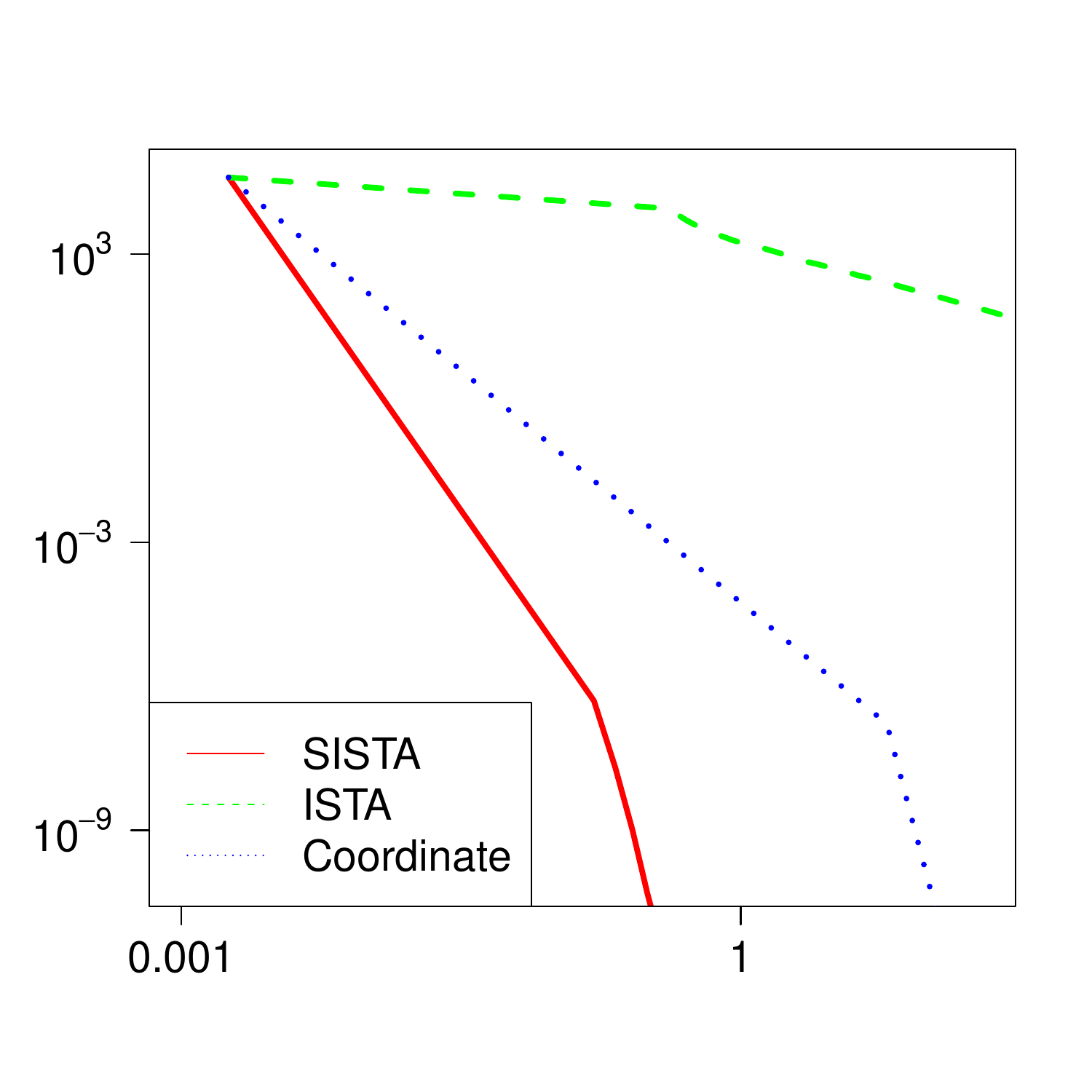}
\end{subfigure}\hspace{-0.15in}
\begin{subfigure}[t]{0.26\textwidth}
\centering
\includegraphics[width=\linewidth]{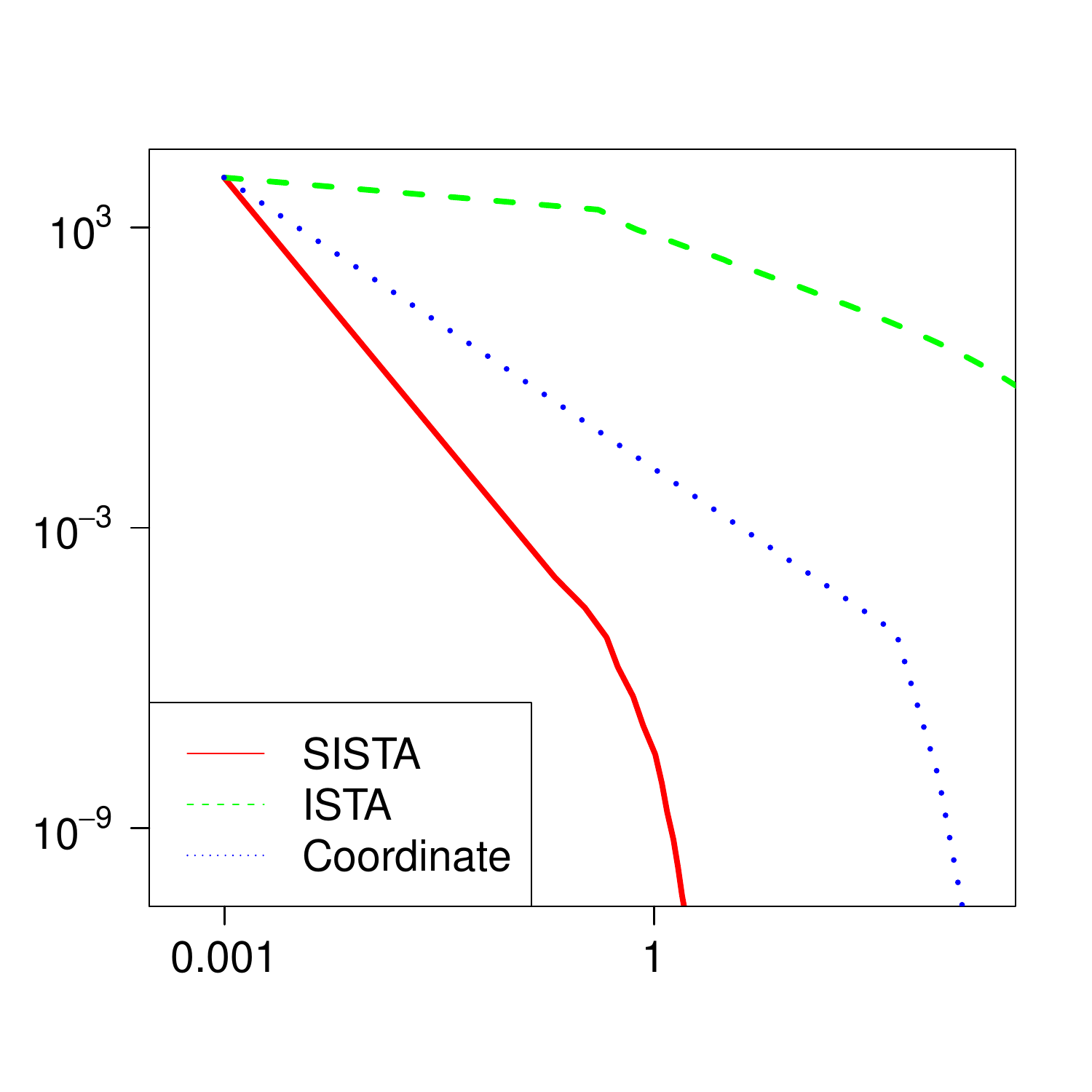}
\end{subfigure}\hspace{-0.15in}
\begin{subfigure}[t]{0.26\textwidth}
\centering
\includegraphics[width=\linewidth]{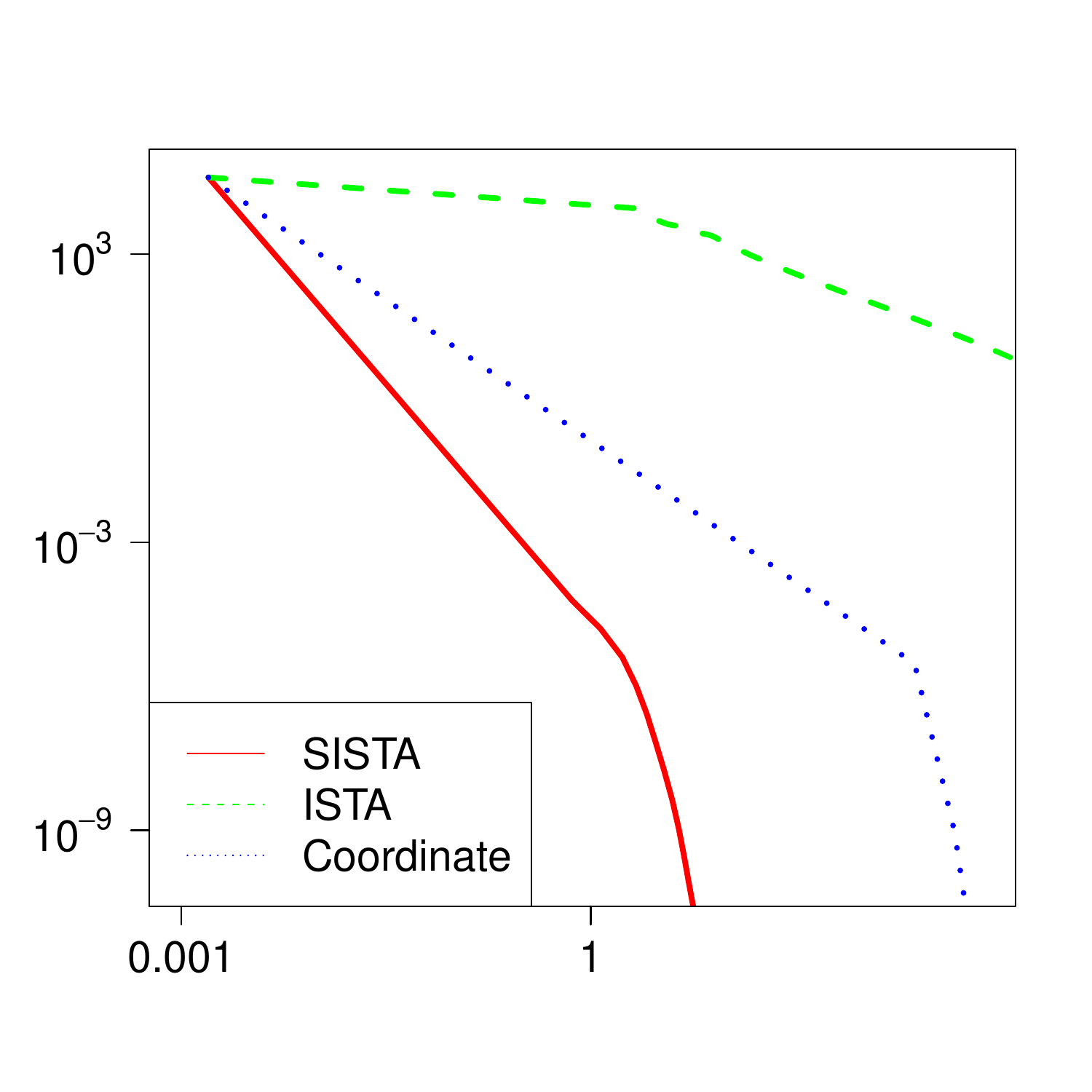}
\end{subfigure}
\begin{subfigure}[t]{0.26\textwidth}
\centering
\includegraphics[width=\linewidth]{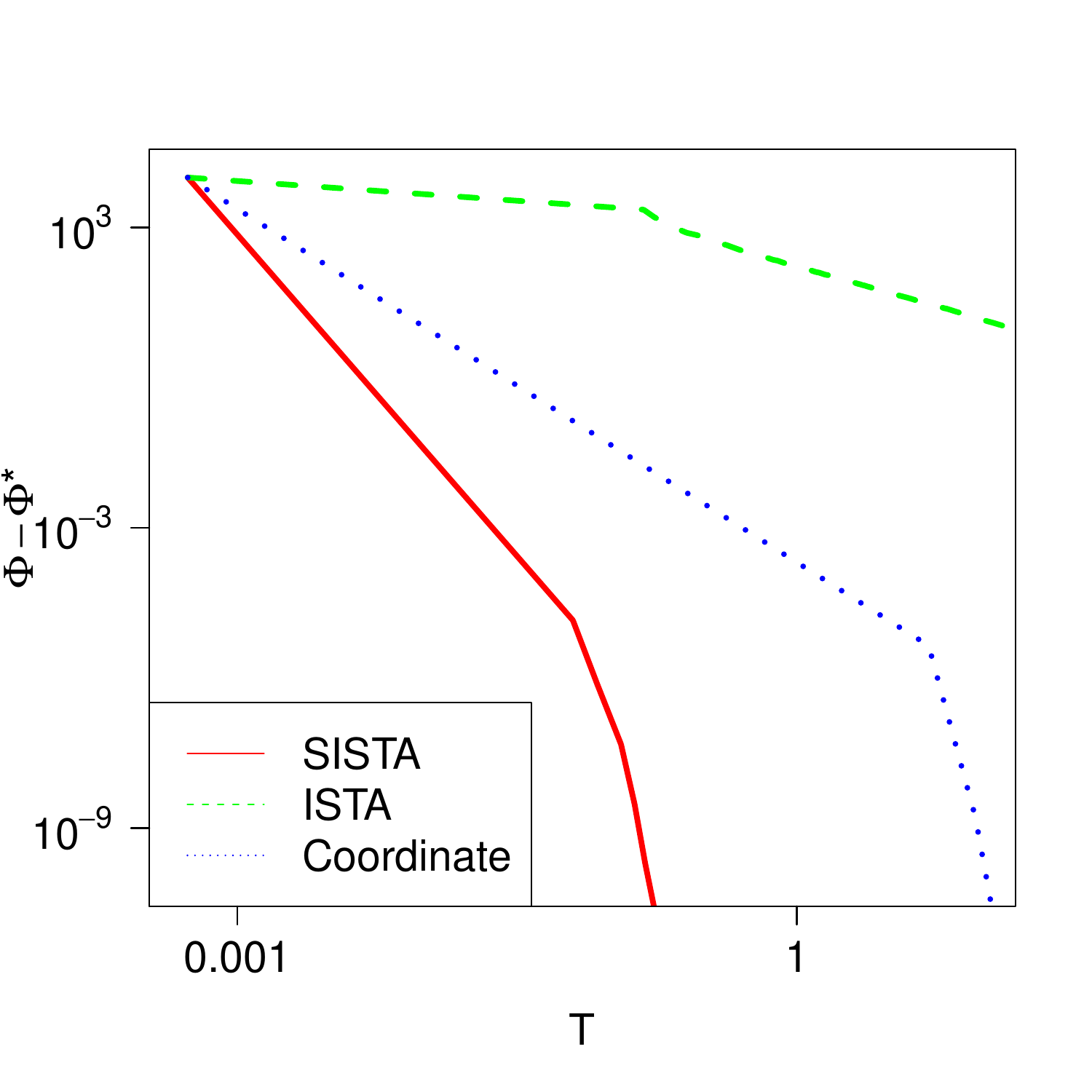}
\caption{$K=100, N=100$}
\end{subfigure}\hspace{-0.15in}
\begin{subfigure}[t]{0.26\textwidth}
\centering
\includegraphics[width=\linewidth]{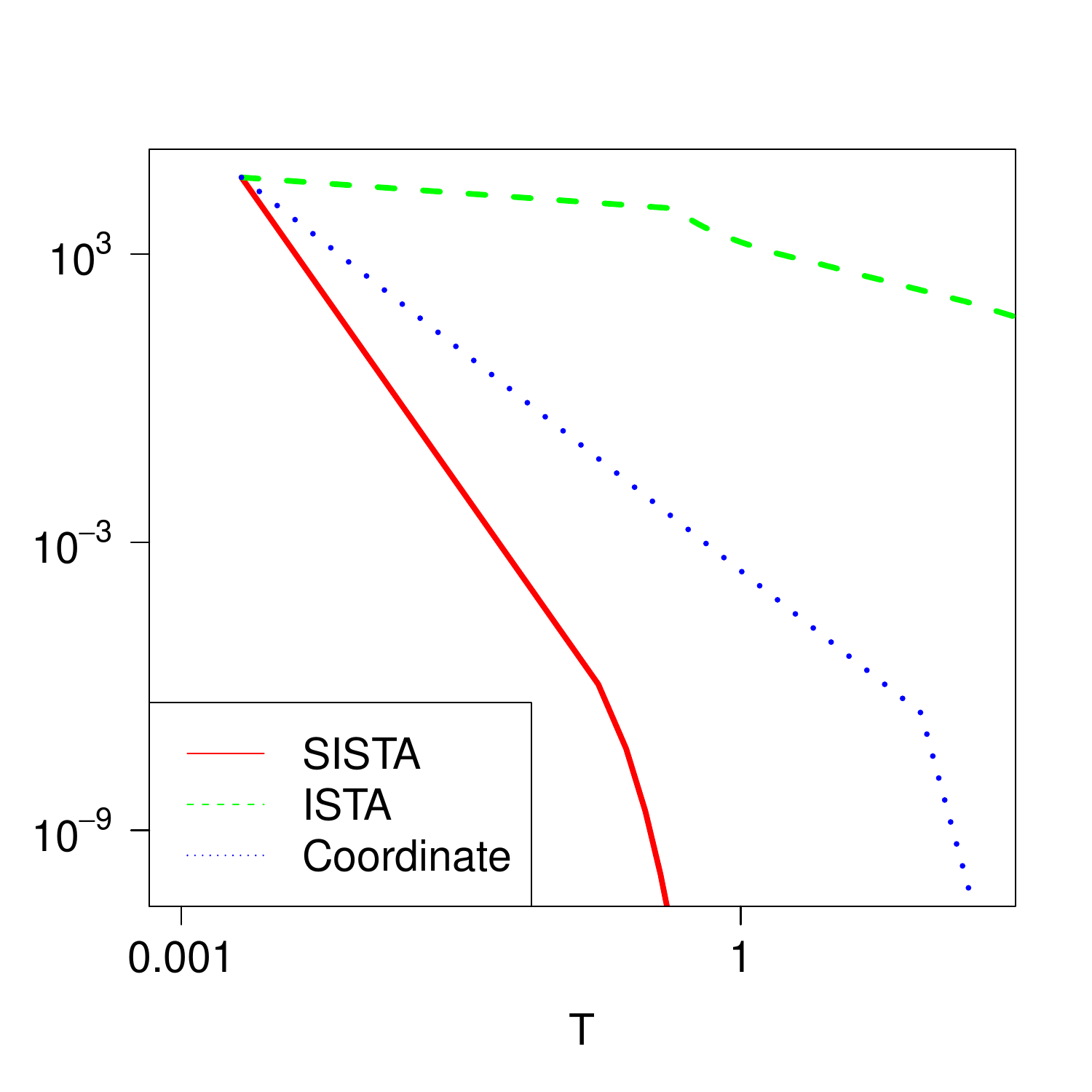}
\caption{$K=100, N=200$}
\end{subfigure}\hspace{-0.15in}
\begin{subfigure}[t]{0.26\textwidth}
\centering
\includegraphics[width=\linewidth]{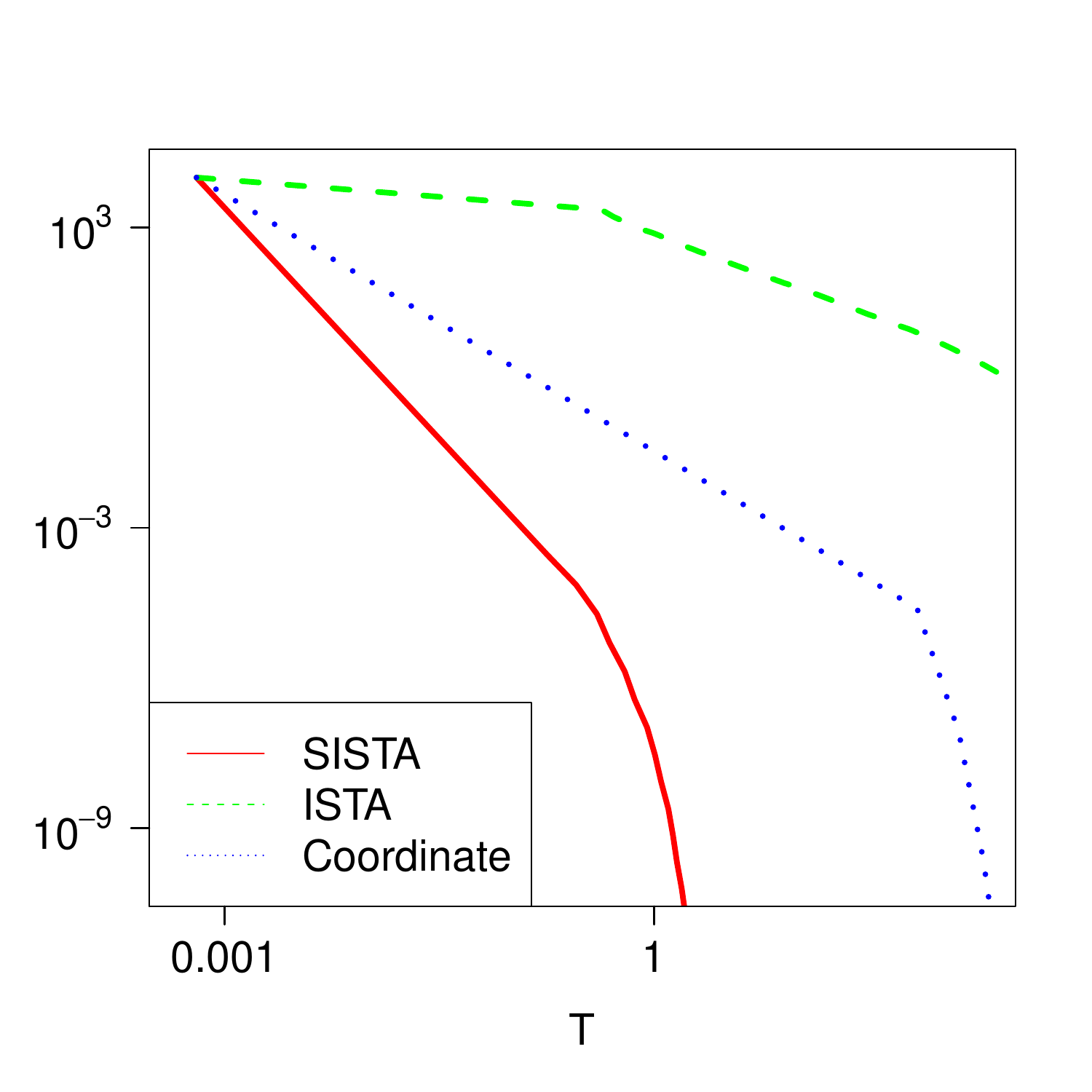}
\caption{$K=500, N=100$}
\end{subfigure}\hspace{-0.15in}
\begin{subfigure}[t]{0.26\textwidth}
\centering
\includegraphics[width=\linewidth]{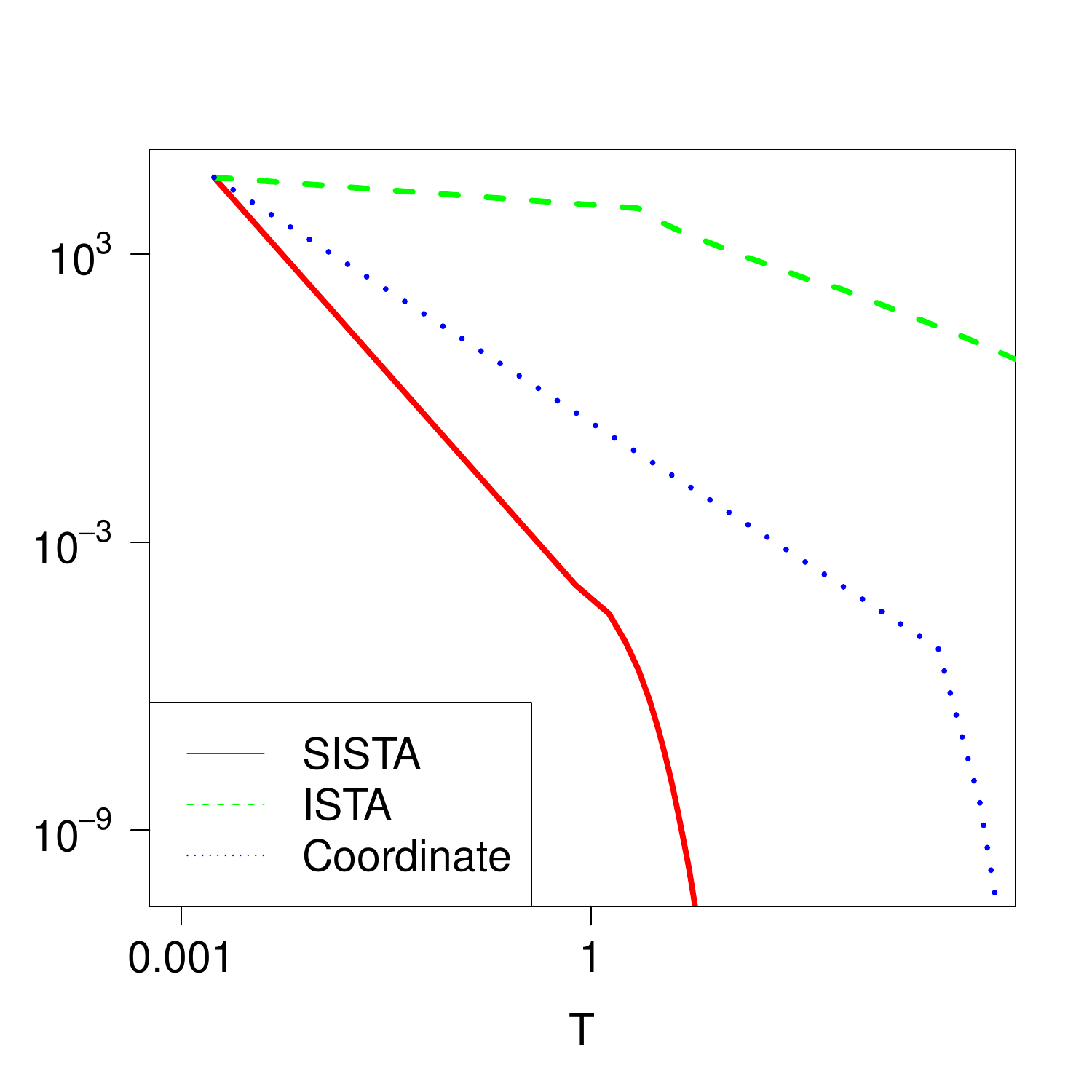}
\caption{$K=500, N=200$}
\end{subfigure}
\caption{Comparison of SISTA, coordinate descent, and ISTA for the specified
problem sizes. The top row has sparsity level $= 0.05$ and the bottom row
has sparsity level $= 0.1$.}
\label{simulations}
\end{figure}

For comparability, the three methods are run using the same initial estimate
of the variables, in practice by setting them all to be zero. We can see
from figure \ref{simulations} that SISTA runs an order of magnitude faster
than the two pure methods it is a hybrid of. The ISTA method has two major
drawbacks compared to the other two. First, a proper initialization of $u,v$
is crucial to ensure a quick convergence of gradient based approaches.
Second, a step size has to be set for $u,v$ in addition to $\beta $, which
adds extra complexities to the algorithm as they do not necessarily have the
same scale as $\beta $. The limitations of coordinate descent are also
straightforward to see. As the dimensionality of $\beta $ increases, which
is common in practical applications, the cost of coordinate descent
increases dramatically, since the algorithm needs to evaluate the gradient
after each component's update, which is time-intensive.
\section{Application to migration flows\label{sec:appli}}

\label{application}

\subsection{Literature}
Our goal in this section is to predict migration flows between countries. In
this application, $\hat{\pi}_{ij}$ is the probability that a migrant drawn
from the overall population of migrants has origin country $i$ and
destination country $j$.

The recent compilation of country-to-country migration data (see \cite%
{Ozden11} for instance) has initiated a fast growing literature in economics
whose aim is to predict bilateral migration flows using measures of
dissimilarity between an origin and a destination country (see \cite{Beine16}
for a recent critical review). The dissimilarity measures are built to
capture various dimensions of attractiveness of a destination country for
migrants from a given origin country, which can be classified into two broad classes:
those that are constructed using country-specific characteristics such as Examples (a) and (b) in section \ref{paramCost} and pairwise measures such as Example (c).

The current literature has only considered a limited number of dissimilarity
measures so far -- a minimum of five in \cite{Beine15} and up to a maximum
of eight in \cite{Ramos17}. The considered ones generally include
geographic distances, economic gaps, differences in immigration policy, and
cultural differences (see \cite{Beine15, Beine16b, Belot12, Grogger11,
Ramos17}). The parameters are estimated using
classical non-linear panel data techniques, where the dependent variable is the
logarithm of the bilateral migration flow, and the dissimilarity measures act as
independent variables. Countries of origin and destination fixed-effects are
usually included in the regression to account for the potential bias due to
the fact that the attractiveness of a destination country is relative to
that of alternative destination countries, the so-called \textquotedblleft
multilateral resistance.\textquotedblright

However, the relatively small number of dissimilarity measures studied in
this literature is in sharp contrast with the fact that countries differ
along many dimensions. A recent effort from various organizations including
the World Bank Group, the Centre d'\'{E}tudes Prospectives et d'Informations
Internationales (CEPII), Freedom House, has greatly increased the range
of measures available for countries. Combining these various sources of data
enables one to distinguish more than $200$ countries along over $100$ dimensions,
significantly augmenting the size of dissimilarity measures from five or eight.

While this offers scope for improvement in the prediction of bilateral
migration flows, it also raises the questions of how to select the most relevant
measures of dissimilarity and how to estimate their corresponding parameters. The methodology outlined in the previous section contributes to this literature by providing a solution to them.

\subsection{Data}

Our application requires access to a working dataset containing information
about bilateral migration flows. Each unit in our dataset is an origin-destination country pair. We compile the dataset using five different sources, which gives us over 40,000 observations of bilateral migration flows and more than 100 measures of dissimilarity, including nine pairwise measures.

We use data about bilateral migration flows between 1990 and 2000. These data are obtained in two steps. First, we
obtain data about migration stocks in 1990 and 2000 from the dataset compiled by \cite{Ozden11}, which were collected using census records from 226 countries of origin
and destination. Second, we follow the literature (e.g. \cite{Beine16b}) and calculate
bilateral migration flows between 1990 and 2000 as the change in bilateral stocks of migrants
between the two decades.

\cite{Beine16b} argues that the existence of a social network at destination, defined as the stock of migrants from the same origin, explains the size and the patterns of international migration flows. We follow \cite{Beine16b} and use the stock of migrants from origin country $i$ in destination country $j$ in 1990 as a proxy measure of the network of migrants from $i$ in $j$  in 2000.

We obtain the second set of dissimilarity measures using the data compiled by CEPII and documented in \cite{Mayer11}. These data cover 223 countries and contain eight pairwise measures of dissimilarity, including the geographic distance between two countries, whether they are contiguous, and whether they have had a colonial link after 1945.










We incorporate three additional sources to append information about country-specific
characteristics covering a vast range of topics.

From the World Development Indicators (WDI), compiled by the World Bank, we
obtain 73 variables containing information about economic development
(GDP, income, unemployment, inflation, human capital), demography, health
(life expectancy, access to improved sanitation facilities), and environment (climate, pollution) for 217 countries.

From the Worldwide Governance Indicators (WGI), also compiled by the World Bank,
we obtain six variables about the quality of governance:
voice and accountability, political stability and absence of violence,
government effectiveness, regulatory quality, rule of law, and control of
corruption for 213 countries.

Finally, from Freedom in the World Indicators, compiled by the Freedom
House organization, we obtain another 15 variables providing information
about political rights and civil liberties for 210 countries.

Combining all these variables, we end up with a vector of 94 characteristics
for each country. We then construct
94 measures of dissimilarity as in Example (a) in section \ref{paramCost}. 
\subsection{Results}
We apply SISTA to our working dataset complied in the previous section.
Since the current literature uses between five and eight dissimilarity
measures, we perform two sets of experiments, where we conduct a grid search
of $\gamma$ so that five and eight components are non-zero in the learned $%
\beta$.

\begin{table}[H]
\caption{Learned $\protect\beta$ with (top) five (bottom) eight non-zero
components.}
\label{table}\centering
\begin{adjustbox}{max width=\textwidth}
\begin{tabular}{llllllll}
\toprule Contiguity & \makecell{Common\\ language} & \makecell{Colonial
link\\ after WWII} & \makecell{(log) Geo. \\ distance} & Network & \makecell{Improved sanitation \\ facilities, urban\\Ori.*Dest.} & \makecell{Area \\Ori.*Dest.} & %
\makecell{Life expect. female \\Ori.*Dest.} \\
\midrule 1.797 (0.057) & 0 & 1.224 (0.194) & -1.540 (0.020) & 1.309 (0.178) & 0  & -0.630 (0.043) & 0
\\
\midrule 2.438 (0.058) & 4.412 (0.108) & 2.118 (0.431) & -1.992 (0.048) & 1.633 (0.211) & 2.801 (0.318) & -0.927 (0.062) & 4.202 (0.535) \\
\bottomrule &  &  &  &  &  &  &
\end{tabular}%
\end{adjustbox}
\par
{\footnotesize Note: Standard errors, calculated using 1,000 bootstrapped
samples, are in parentheses. }
\end{table}

Table \ref{table} shows the learned $\beta $ from the two experiments. The
top row indicates that if only five measures of dissimilarity are allowed in
the prediction of migration flows, three of the selected ones are used in most papers in this literature: whether two
countries are contiguous, logarithm of geographic distance between them, and whether they have had a colonial link after World
War II. The fourth one is the network measure introduced in
\cite{Beine15}. Moreover, the signs of $\beta_k$'s corresponding to these four measures  match precisely with previous results. Interestingly,
the fifth selected one has never been used. It
corresponds to the interaction between the areas of the origin and destination countries. The negative sign of the parameter indicates that migration flows are larger between small (vast) origin countries and vast
(resp. small) destination countries compared to origin-destination countries with similar areas.

The second row indicates that if only eight measures of dissimilarity are to
be selected, besides the five mentioned above, the three additional
ones are: whether two countries share a language spoken by at least nine
percent of their respective population, the interaction between the share of
urban population that has access to improved sanitation facilities, and the interaction between the female life expectancy at birth in the origin and destination countries. While the first measure has already been used in \cite{Ramos17}, with its
parameter having the same sign as ours, the other two have not appeared in the literature. The parameters for these new measures are positive,
indicating that migration flows tend to be larger between countries with
similar urban access to improved sanitation facilities and female life
expectancy at birth.

\section{Conclusion\label{sec:conc}}
This paper introduces a new algorithm, which we called SISTA, to learn the
transport cost in optimal transport problems. In this type of problems, one
needs to optimize simultaneously over the potentials and over the parameters
of the cost. As the parameter vector is sparse, we add an $l_{1}$
penalization over the parameters. SISTA alternates between a phase of exact
minimization over the transport potentials and a phase of proximal gradient
descent over the parameters of the transport cost. We prove its linear
convergence and illustrate through numerical experiments its rapid
convergence compared to coordinate descent and ISTA. In our application of
predicting bilateral migration flows, SISTA allows us to learn which
measures of dissimilarity between an origin and a destination country are
the most important ones. Our approach reveals that dissimilarities between
origin and destination countries in terms of area, female life expectancy,
and urban access to improved sanitation facilities, are critical predictors
of bilateral migrations flows that have been absent from this literature.

Our method applies to a broad range of problems in quantitative social
sciences. We present an application to the prediction of bilateral migration
flows using country-specific characteristics and pairwise measures of
dissimilarity between countries. The same technique could be applied to
predicting bilateral trade flows, the matching of workers to jobs, men to
women, children to schools, etc. In all these applications, there exists a
long list of attributes/characteristics upon which measures of
dissimilarity between countries of origin and countries of destination, men
and women, workers and jobs can be constructed and could explain flows or
matches. Our approach allows one to select those that matter the most.

More generally, the idea underlying the SISTA algorithm is broadly
applicable. In many other optimization problems, it may be worthwile using
hybrid methods combining the strengths of several existing descent methods.
As exemplified by SISTA, the hybrid version can be much more efficient than the
methods it combines.

\end{document}